\documentclass[11pt]{amsart}

\usepackage{latexsym}
\usepackage{amsmath}
\usepackage{amsfonts}
\usepackage{amssymb}
\usepackage{enumitem}
\usepackage {tikz}
\usepackage{color}
\usetikzlibrary {positioning}
\usetikzlibrary{arrows,shapes}
\usepackage{color}

\usepackage[noend]{algorithmic} 
\usepackage{algorithm,caption}
\algsetup{indent=2em}

\newtheorem{theorem}{Theorem}[section]
\newtheorem{lemma}[theorem]{Lemma}
\newtheorem{corollary}[theorem]{Corollary}
\newtheorem{proposition}[theorem]{Proposition}

\newtheorem{sublemma}{}[theorem]

\theoremstyle{definition}

\theoremstyle{remark}

\numberwithin{equation}{section}
\usepackage{stackengine}

\stackMath

\newcommand{\ba}{\backslash}

\begin{document}

\title[Cographs and $1$-sums]{Cographs and $1$-sums}


\author{Jagdeep Singh}
\address{Department of Mathematics and Statistics\\
Binghamton University\\
Binghamton, New York}
\email{jsingh@binghamton.edu}

\subjclass{05C40, 05C75, 05C83}
\date{\today}

\begin{abstract}
A graph that can be generated from $K_1$ using joins and $0$-sums is called a cograph. We define a sesquicograph to be a graph that can be generated from $K_1$ using joins, $0$-sums, and $1$-sums. We show that, like cographs, sesquicographs are closed under induced minors. Cographs are precisely the graphs that do not have the 4-vertex path as an induced subgraph. We obtain an analogue of this result for sesquicographs, that is, we find those non-sesquicographs for which every proper induced subgraph is a sesquicograph. 
\end{abstract}

\maketitle

\section{Introduction}
\label{intro}
In this paper, we only consider finite and simple graphs. The notation and terminology follows \cite{text} except where otherwise indicated. For graphs $G$ and $H$ having disjoint vertex sets, the {\bf $0$-sum} $G \oplus H$ of $G$ and $H$ is their disjoint union. A {\bf $1$-sum} $G \oplus_1 H$ of $G$ and $H$ is obtained by identifying a vertex of $G$ with a vertex of $H$. The {\bf join} $G \bigtriangledown H$ of two disjoint graphs $G$ and $H$ is obtained from the $0$-sum of $G$ and $H$ by joining every vertex of $G$ to every vertex of $H$. A {\bf cograph} is a graph that can be generated from the single- vertex graph $K_1$ using the operations of join and $0$-sum. We define a graph to be a {\bf sesquicograph} if it can be generated from $K_1$ using the operations of join, $0$-sum, and $1$-sum. The class of cographs has been extensively studied over the last fifty years (see, for example, \cite{corneil2, jung, sein}). Due to the following characterization, cographs are also called $P_4$-free graphs \cite{corneil}.

\begin{theorem}
\label{cographs_characterisation}
A graph $G$ is a cograph if and only if $G$ does not contain the path $P_4$ on four vertices as an induced subgraph.
\end{theorem}

Since we consider only simple graphs in this paper, when we write $G/e$ for an edge $e$ of a graph $G$, we mean the simple graph obtained from the multigraph that results from contracting the edge $e$ by deleting all but one edge from each class of parallel edges. An {\bf induced minor} of a graph $G$ is a graph $H$ that can be obtained from $G$ by a sequence of operations each consisting of a vertex deletion or an edge contraction. In Section $2$, we show that every induced minor of a sesquicograph is a sesquicograph. In addition, we provide an alternative definition of a sesquicograph in terms of the vertex connectivities of its induced subgraphs and their complements. The graph obtained from a $6$-cycle by adding a chord to create two $4$-cycles is called the {\bf domino} graph. We let $C_6^+$ denote the domino; $\overline{P_5}$ is the complement of a $5$-vertex path. The next theorem is the main result of the paper.

\begin{theorem}
\label{main}
A graph $G$ is a sesquicograph if and only if $G$ does not contain any of the following graphs as an induced subgraph:

\begin{enumerate}[label=(\roman*)]
    \item cycles of length exceeding four, and
    \item $\overline{P_5}, C_6^+, H_1, H_2, H_3, H_4,$ and $H_5$,
\end{enumerate}

where the graphs in $(ii)$ are shown in Figure $\ref{Figure}$.
\end{theorem}

Its proof occupies most of Section $3$. As a consequence of Theorem \ref{main}, we have the following characterization of sesquicographs in terms of forbidden induced minors.

\begin{corollary}
\label{main_corollary}
A graph $G$ is a sesquicograph if and only if $G$ has no induced minor isomorphic to a graph in $\{C_5, \overline{P_5}, H_1, H_2, H_3, H_4, H_5\}$, where $C_5$ is the cycle of length five.
\end{corollary}

A graph $G$ is a \textbf{$2$-cograph} if it can be generated from $K_1$ using the operations of complementation, $0$-sum, and $1$-sum. The class of $2$-cographs has been studied in \cite{oxjag}. This paper has some similarities with \cite{oxjag} although the arguments for sesquicographs are not as complex as they are for $2$-cographs. Since the class of sesquicographs is the smallest class of graphs that contains $K_1$ and is closed under the operations of join, $0$-sum, and $1$-sum, it is a proper subclass of $2$-cographs and, thus, of the class of perfect graphs. Note the path $P_5$ on five vertices is a sesquicograph but its complement $\overline{P_5}$ is not. It follows that the class of sesquicographs is not closed under complementation unlike the classes of cographs and $2$-cographs.

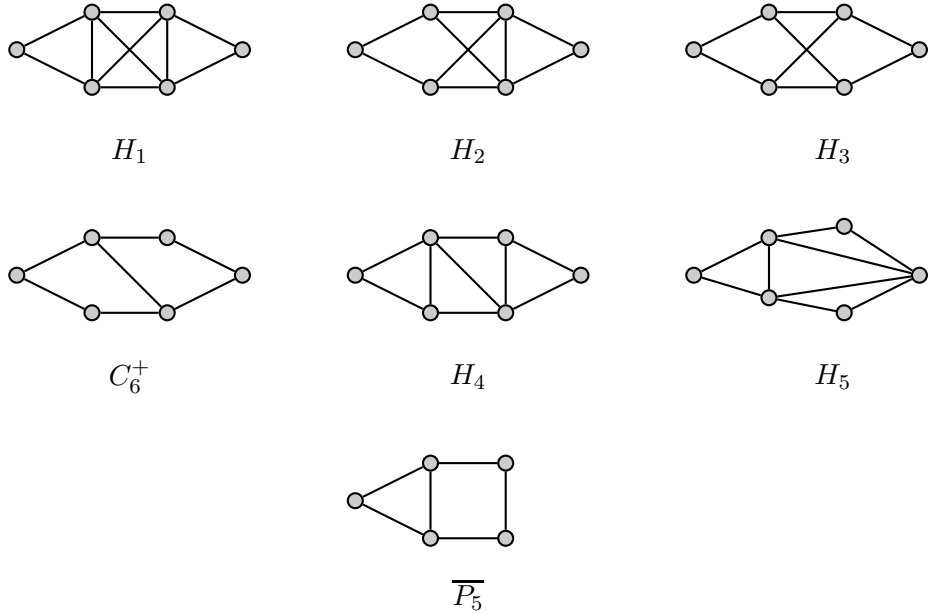
\begin{figure}
    \centering
    \begin{tikzpicture}[scale=0.5,colorstyle/.style={circle, draw=black!100,fill=black!20, thick, inner sep=2pt, minimum size=0.5mm}]

    \node (a1) at (-18,6.2)[colorstyle]{};
    \node (a2) at (-16,7.2)[colorstyle]{};
    \node (a3) at (-14,7.2)[colorstyle]{};
    \node (a4) at (-12,6.2)[colorstyle]{};
    \node (a6) at (-16,5.2)[colorstyle]{};
    \node (a5) at (-14,5.2)[colorstyle]{};
    
    \draw[thick] (a1)--(a2)--(a3)--(a4)--(a5)--(a6)--(a1);
    
    \draw[thick] (a2)--(a5)--(a3)--(a6)--(a2);
    
    \node (a) at (-15,3.5){$H_1$};

    \node (b1) at (-9,6.2)[colorstyle]{};
    \node (b2) at (-7,7.2)[colorstyle]{};
    \node (b3) at (-5,7.2)[colorstyle]{};
    \node (b4) at (-3,6.2)[colorstyle]{};
    \node (b6) at (-7,5.2)[colorstyle]{};
    \node (b5) at (-5,5.2)[colorstyle]{};
    
    \draw[thick] (b1)--(b2)--(b3)--(b4)--(b5)--(b6)--(b1);
    
    \draw[thick] (b2)--(b5)--(b3)--(b6);
    
    \node (b) at (-6,3.5){$H_2$};

    \node (c1) at (0,6.2)[colorstyle]{};
    \node (c2) at (2,7.2)[colorstyle]{};
    \node (c3) at (4,7.2)[colorstyle]{};
    \node (c4) at (6,6.2)[colorstyle]{};
    \node (c6) at (2,5.2)[colorstyle]{};
    \node (c5) at (4,5.2)[colorstyle]{};
    
    \draw[thick] (c1)--(c2)--(c3)--(c4)--(c5)--(c6)--(c1);
    
    \draw[thick] (c2)--(c5);
    \draw[thick] (c3)--(c6);
    
    \node (c) at (3.7,3.5){$H_3$};

     \node (d1) at (-18,0.2)[colorstyle]{};
    \node (d2) at (-16,1.2)[colorstyle]{};
    \node (d3) at (-14,1.2)[colorstyle]{};
    \node (d4) at (-12,0.2)[colorstyle]{};
    \node (d6) at (-16,-0.8)[colorstyle]{};
    \node (d5) at (-14,-0.8)[colorstyle]{};
    
    \draw[thick] (d1)--(d2)--(d3)--(d4)--(d5)--(d6)--(d1);
    \draw[thick] (d2)--(d5);
    
    \node (d) at (-15,-2.5){$C_6^+$};

    \node (e1) at (-9,0.2)[colorstyle]{};
    \node (e2) at (-7,1.2)[colorstyle]{};
    \node (e3) at (-5,1.2)[colorstyle]{};
    \node (e4) at (-3,0.2)[colorstyle]{};
    \node (e6) at (-7,-0.8)[colorstyle]{};
    \node (e5) at (-5,-0.8)[colorstyle]{};
    
    \draw[thick] (e1)--(e2)--(e3)--(e4)--(e5)--(e6)--(e1);
    
    \draw[thick] (e6)--(e2)--(e5)--(e3);
    
    \node (e) at (-6,-2.5){$H_4$};

    \node (f1) at (0,0.2)[colorstyle]{};
    \node (f2) at (2,1.2)[colorstyle]{};
    \node (f3) at (4,1.5)[colorstyle]{};
    \node (f4) at (6,0.2)[colorstyle]{};
    \node (f6) at (2,-0.4)[colorstyle]{};
    \node (f5) at (4,-0.8)[colorstyle]{};
    
    \draw[thick] (f1)--(f2)--(f3)--(f4)--(f5)--(f6)--(f1);
    
    \draw[thick] (f2)--(f6)--(f4)--(f2);
    
    \node (f) at (3.7,-2.5){$H_5$};

    \node (g1) at (-9,-5.8)[colorstyle]{};
    \node (g2) at (-7,-4.8)[colorstyle]{};
    \node (g3) at (-5,-4.8)[colorstyle]{};
    \node (g6) at (-7,-6.8)[colorstyle]{};
    \node (g5) at (-5,-6.8)[colorstyle]{};
    
    \draw[thick] (g1)--(g2)--(g3)--(g5)--(g6)--(g1);
    \draw[thick] (g2)--(g6);
    
    \node (g) at (-6,-8.3){$\overline{P_5}$};
    
\end{tikzpicture}
\caption{The induced-subgraph-minimal non-sesquicographs.}
\label{Figure}
\end{figure}

\section{Preliminaries}

Let $G$ be a graph. A vertex $u$ of $G$ is a {\bf neighbour} of a vertex $v$ of $G$ if $uv$ is an edge of $G$.  The {\bf neighbourhood $N_G(v)$} of $v$ in $G$ is the set of all neighbours of $v$ in $G$. If $G$ is connected, a {\bf $t$-cut} of $G$ is set $X_t$ of vertices of $G$ such that $|X_t| = t$ and $G - X_t$ is disconnected. A graph that has no $t$-cuts for all $t$ less than $k$ is {\bf $k$-connected}.
Viewing $G$ as a subgraph of $K_n$ where $n = |V(G)|$, we colour the edges of $G$ green while assigning the colour red to the non-edges of $G$. Similar to the terminology in \cite{oxjag}, we use the terms {\bf green graph} and {\bf red graph} for $G$ and its complementary graph $\overline{G}$, respectively. An edge of $G$ is called a {\bf green edge} while a {\bf red edge} refers to an edge of $\overline{G}$. The {\bf green degree} of a vertex $v$ of $G$ is the number of {\bf green neighbours} of $v$, while the {\bf red degree} of $v$ is its number of {\bf red neighbours}.

We omit the straightforward proofs of the next three results.

\begin{lemma}
\label{small_sesquicographs}
All graphs having at most four vertices are sesquicographs.
\end{lemma}

\begin{lemma}
\label{join_observation}
A graph $G$ is a join of two graphs if and only if its complement $\overline{G}$ is disconnected.
\end{lemma}

\begin{lemma}
\label{complementation_contraction_together}
Let $G$ be a graph and let $uv$ be an edge $e$ of $G$. Then  $\overline{G/e}$ is the graph obtained by adding a vertex $w$ with neighbourhood $N_{\overline{G}}(u) \cap N_{\overline{G}}(v)$ to the graph $\overline{G} - \{u,v\}$.
\end{lemma}

\begin{lemma}
\label{closed_induced_subgraph}
Every induced subgraph of a sesquicograph is a sesquicograph.
\end{lemma}

\begin{proof}
Let $G$ be a sesquicograph. It is enough to show that, for every vertex $v$ of $G$, the graph $G-v$ is a sesquicograph. Note that if $|V(G)| \leq 5$, then our result follows by Lemma \ref{small_sesquicographs}. Let $|V(G)| = n$. We proceed via induction on $|V(G)|$ and assume that the result is true for all sesquicographs with order less than $n$. Since $G$ is a sesquicograph, $G$ is a $0$-sum, a $1$-sum, or a join of proper induced subgraphs $X$ and $Y$ of $G$. Observe that if $G$ is $X \oplus Y$ or $X \bigtriangledown Y$, then $G-v$ equals $(X-v) \oplus Y$ or $(X-v) \bigtriangledown Y$, and so the result follows by induction. Therefore we may assume that $G = X \oplus_1 Y$. Note that, in this case, $G-v$ is either $(X-v) \oplus (Y-v)$ or $(X-v) \oplus_1 Y$. Thus our result follows by induction.
\end{proof}

A graph is {\bf trivial} if it contains only one vertex and no edge. Cographs can also be characterized as the graphs in which every non-trivial connected induced subgraph has a disconnected complement. Similarly, a graph $G$ is a $2$-cograph if $G$ has no non-trivial induced subgraph $H$ such that both $H$ and $\overline{H}$ are $2$-connected. Next we show that sesquicographs can be characterized in a similar way.

\begin{proposition}
\label{alternate_definition}
A graph $G$ is a sesquicograph if and only if, for every non-trivial induced subgraph $H$ of $G$, the graph $H$ is not $2$-connected or $\overline{H}$ is disconnected.
\end{proposition}

\begin{proof}
Let $G$ be a sesquicograph and let $H$ be a non-trivial induced subgraph of $G$. By Lemma \ref{closed_induced_subgraph}, $H$ is a sesquicograph. Since $H$ can be decomposed as a $0$-sum, a $1$-sum, or a join, it follows by Lemma \ref{join_observation}, that $H$ is not $2$-connected or $\overline{H}$ is disconnected.

Conversely, let $G$ be a graph such that, for every non-trivial induced subgraph $H$ of $G$, the graph $H$ is not $2$-connected or $\overline{H}$ is disconnected. By Lemma \ref{join_observation}, it follows that every non-trivial subgraph of $G$ can be written as a $0$-sum, a $1$-sum, or a join of smaller induced subgraphs of $G$. Therefore $G$ can be generated from $K_1$ using the operations of $0$-sum, $1$-sum, and join. Thus $G$ is a sesquicograph.
\end{proof}


A slight variation of the proof of the closure of $2$-cographs under contractions  \cite[Proposition $2.8$]{oxjag} shows that sesquicographs are also closed under contractions.

\begin{proposition}
\label{closed_under_contractions}
Let $G$ be a sesquicograph and $e$ be an edge of $G$. Then $G/e$ is a sesquicograph.
\end{proposition}

\begin{proof}
Assume to the contrary that $G/e$ is not a sesquicograph. Then there is a non-trivial induced subgraph $H$ of $G/e$ such that $H$ is $2$-connected and $\overline{H}$ is connected. Let $e = uv$ and let $w$ denote the vertex in $G/e$ obtained by identifying $u$ and $v$. We may assume that $w$ is a vertex of $H$, otherwise $H$ is an induced subgraph of $G$, a contradiction. We assert that the subgraph $H'$ of $G$ induced on the vertex set $(V(H) \cup \{u,v\}) - \{w\}$ is $2$-connected and its complement $\overline{H'}$ is connected. To see this, note that, since $H$ is $2$-connected, $H'$ is $2$-connected unless one of $u$ and $v$, say $u$, is a leaf of $H'$. In the exceptional case, we have $H' - u \cong H$, so $G$ has a $2$-connected induced subgraph for which its complement is connected, a contradiction. We deduce that $H'$ is $2$-connected. 

Note that, by Lemma \ref{complementation_contraction_together}, $\overline{H}$ is obtained from $\overline{H'}$ by adding a vertex $w$ with neighbourhood $N_{\overline{H'}}(u) \cap N_{\overline{H'}}(v)$ to the graph $\overline{H'}-\{u,v\}$. Since $\overline{H}$ is connected, it follows that $\overline{H'}$ is connected, a contradiction.
\end{proof}

It now follows that the class of sesquicographs is closed under taking induced minors.
Since we can compute the components and blocks of a graph in  polynomial time \cite[4.1.23]{west}, the algorithm in Figure \ref{identify} recognizes sesquicographs in polynomial time.

\begin{figure}
\begin{algorithmic}[]
\REQUIRE Input a simple graph $G$
\STATE Set $H \leftarrow G$, BlocksList $\leftarrow [G]$

\IF{$|V(H)| \leq 4$}

\STATE remove $H$ from BlocksList

\IF{BlocksList is empty}

\STATE return $G$ is a sesquicograph and exit the algorithm

\ELSE

\STATE update $H$ to be an element of BlocksList

\ENDIF
\ENDIF

\IF{$H$ is not $2$-connected}

\STATE remove $H$ from BlocksList
\STATE Decompose $H$ into $2$-connected blocks and add all the blocks of $H$ to BlocksList
\STATE update $H$ to be an element of BlocksList

\ELSIF{$\overline{H}$ is not connected}

\STATE remove $H$ from BlocksList
\STATE Decompose $\overline{H}$ into connected components and add the complements of all the components to BlocksList
\STATE update $H$ to be an element of BlocksList

\ELSE

\STATE return $G$ is not a sesquicograph and exit the algorithm

\ENDIF

\end{algorithmic}
\caption{Algorithm for recognizing a sesquicograph.}
\label{identify}
\end{figure}

\section{Induced-subgraph-minimal non-sesquicographs}

We noted in Section $2$ that sesquicographs are closed under induced subgraphs. In this section, we consider those non-sesquicographs for which every proper induced subgraph is a sesquicograph. We call these graphs {\bf induced-subgraph-minimal non-sesquicographs}. The goal of this section is to characterize such graphs. We begin by showing that all cycles of length exceeding four are examples of such graphs.

\begin{lemma}
\label{cycles_are_minimal_graphs}
Let $G$ be a cycle of length exceeding four. Then $G$ is an induced-subgraph-minimal non-sesquicograph.
\end{lemma}

\begin{proof}
Note that both $G$ and $\overline{G}$ are $2$-connected and so, by Proposition \ref{alternate_definition}, $G$ is not a sesquicograph. It is now enough to show that, for any vertex $v$ of $G$, the graph $G-v$ is a sesquicograph. Observe that $G-v$ is a path and so is a sesquicograph.
\end{proof}

The next result can be easily checked. 

\begin{lemma}
\label{other_minimal_graphs}
The graphs $\overline{P_5}, C_6^+, H_1, H_2, H_3, H_4,$ and $H_5$ are induced-subgraph-minimal non-sesquicographs.
\end{lemma}

\begin{lemma}
\label{minimal_graphs_connectivity}
Let $G$ be an induced-subgraph-minimal non-sesquicograph. Then $G$ is $2$-connected and $\overline{G}$ is connected.
\end{lemma}

\begin{proof}
Assume the contrary. Then for some proper induced subgraphs $X$ and $Y$ of $G$, we can decompose $G$ as $X \oplus Y,$ as $X \oplus_1 Y$, or, by Lemma \ref{join_observation}, as $X \bigtriangledown Y$. Since $G$ is an induced-subgraph-minimal non-sesquicograph, both $X$ and $Y$ are sesquicographs. It now follows that $G$ is a sesquicograph, a contradiction.
\end{proof}

A $2$-connected graph $H$ is \textbf{critically $2$-connected} if $H-v$ is not $2$-connected for all vertices $v$ of $H$.

\begin{lemma}
\label{two_cases}
Let $G$ be an induced-subgraph-minimal non-sesquicograph. Then $G$ is critically $2$-connected, or $G$ has vertex connectivity two and $\overline{G}$ has vertex connectivity one.
\end{lemma}

\begin{proof}
Note that, by Lemma \ref{minimal_graphs_connectivity}, $G$ is $2$-connected and $\overline{G}$ is connected, and, by Proposition \ref{alternate_definition}, for each vertex $v$ of $G$, the graph $G-v$ is not $2$-connected or $\overline{G}-v$ is disconnected. Observe that $G$ has a vertex $v$ such that $\overline{G}-v$ is connected and so $G-v$ is not $2$-connected. Therefore $G$ has vertex connectivity two. Suppose that $G$ is not critically $2$-connected. Then there is a vertex $w$ of $G$ such that $G-w$ is $2$-connected and so $\overline{G}-w$ is disconnected. Therefore the vertex connectivity of $\overline{G}$ is one. 
\end{proof}

Next we find those induced-subgraph-minimal non-sesquicographs $G$ such that $G$ is critically $2$-connected. We will use the following result of Nebesky \cite{nebesky}.

\begin{lemma}
\label{nebesky_result}
Let $G$ be a critically $2$-connected graph such that $|V(G)| \geq 6$. Then $G$ has at least two distinct paths of length exceeding two such that the internal vertices of these paths have degree two in $G$.
\end{lemma}

\begin{lemma}
\label{adjacent_nebesky}
Let $G$ be an induced-subgraph-minimal non-sesquicograph such that $G$ is not isomorphic to a cycle and let $wxyz$ be a path $P$ of $G$ such that both $x$ and $y$ have degree two in $G$. Then $w$ and $z$ are adjacent.
\end{lemma}

\begin{proof}
Assume that $w$ and $z$ are not adjacent. By Lemma \ref{minimal_graphs_connectivity}, $G$ is $2$-connected, so there is a path $P'$ joining $w$ and $z$ such that $P$ and $P'$ are internally disjoint. We may assume that $P'$ is a shortest such path. It now follows that $G$ has a cycle $C$ of length exceeding four as an induced subgraph. Since a cycle of length exceeding four is not a sesquicograph, $G = C$, a contradiction.
\end{proof}

\begin{proposition}
\label{critically_2_connected_case}
Let $G$ be an induced-subgraph-minimal non-sesquicograph such that $G$ is critically $2$-connected. Then $G$ is isomorphic to a cycle of length exceeding four or to the domino.
\end{proposition}

\begin{proof}
We may assume that $G$ is not isomorphic to a cycle exceeding four otherwise we have our result. Note that, by Lemma \ref{small_sesquicographs}, $|V(G)| \geq 5$. Since the cycle of length five is the only critically $2$-connected graph on five vertices, we may assume that $|V(G)| \geq 6$. By Lemma \ref{nebesky_result}, $G$ has two distinct paths $P_1 = abcd$ and $P_2=wxyz$ of length three such that their internal vertices have degree two. By Lemma \ref{adjacent_nebesky}, $a$ and $d$ are adjacent, and $w$ and $z$ are adjacent. Consider the graph $G' = G - \{b,c\}$. Note that $G'$ is $2$-connected and so, by Lemma \ref{alternate_definition}, $\overline{G}$ is disconnected. It is now easy to check that $|V(G')|=4$ and so $G$ is isomorphic to the domino.
\end{proof}

\begin{figure}
\renewcommand{\thealgorithm}{}

\begin{algorithmic}
\STATE Set FinalList $\leftarrow \emptyset$, $i \leftarrow 0$
\STATE Generate all two connected graphs of order $6$ using nauty geng \cite{nauty} and store in an iterator $L$

\FOR{$g$ in $L$ such that vertex connectivity of $g$ is $2$ and $\overline{g}$ is $1$}

    \FOR{$v$ in $V(g)$}
        \STATE $h=g \ba v$
        \IF{vertex connectivity of $h < 2$ or vertex connectivity of $\overline{h} < 1$}
            \STATE $i \leftarrow i+1$
        \ENDIF
    \ENDFOR

\IF{$i$ equals $|V(g)|$}

\STATE Add $g$ to FinalList

\ENDIF

\ENDFOR

\end{algorithmic}
\caption{Finding induced-subgraph-minimal non-sesquicographs of order six.}
\label{case_6}
\end{figure}

\begin{proof}[Proof of Theorem \ref{main}]
We may assume that $G$ is not critically $2$-connected otherwise we are done by Proposition \ref{critically_2_connected_case}. By Lemma \ref{minimal_graphs_connectivity}, $G$ has vertex connectivity two and $\overline{G}$ has vertex connectivity one. We first show the following.



\begin{sublemma}
\label{red_cut_vertex}
$\overline{G}$ has at most three cut vertices.
\end{sublemma}

Let $\{u,v\}$ be a $2$-cut of $G$ and let the components of $G-\{u,v\}$ be partitioned into subgraphs $A$ and $B$ such that $|V(A)| \geq |V(B)|$ and $|V(A)|-|V(B)|$ is a minimum. Observe that $\overline{G}-x$ is connected for a vertex $x$ in $V(G)$ unless $x$ is the only red neighbour of $u$ or the only red neighbour of $v$, or $|V(B)|=1$ and $x$ is in $V(B)$. Thus \ref{red_cut_vertex} holds.


We show next that the number of vertices of $G$ can be bounded.

\begin{sublemma}
\label{bound_1}
$|V(G)| \leq 6.$
\end{sublemma}

 Assume that $|V(G)|>6$. By \ref{red_cut_vertex}, $\overline{G}$ has at most three cut vertices. First suppose that $\overline{G}$ has one cut vertex $x$. Let the components of $\overline{G}-x$ be partitioned into subgraphs $R_1$ and $R_2$ such that $|V(R_1)| \geq |V(R_2)|$ and $|V(R_1)|-|V(R_2)|$ is a minimum. Since $|V(G)|\geq 7$, we have $|V(R_1)|\geq 3$. Observe that, if $|V(R_2)|\geq 2$, then there exists a vertex $r$ in $R_1$ such that $x$ has two green neighbours in $G-r$. Note that every edge joining a vertex in $R_1$ to a vertex in $R_2$ is a green edge and so $G-r$ is connected. Since every vertex in $V(G)-x$ is in a green $2$-cut, this is a contradiction. Therefore $|V(R_2)|=1$ and so $|V(R_1)|\geq 5$. Let $R_2 = \{\alpha\}$. Note that $G-x$ is $2$-connected since $G$ is not critically $2$-connected. It is now clear that $G-\{x, \alpha\}$ is connected. If $G-\{x,\alpha\}$ has a vertex $r$ such that $G-\{x,\alpha, r\}$ is connected and contains two green neighbours of $x$, then $G-\alpha$ is $2$-connected, a contradiction. It now follows that $G-\{x, \alpha\}$ is a path and its leaves are the only green neighbours of $x$. Note that $G-\alpha$ is a cycle of length exceeding four, a contradiction. 

Next suppose that $\overline{G}$ has two cut vertices $x_1$ and $x_2$. For $\{i,j\} = \{1,2\}$, let $R_i$ be the disjoint union of the components of $\overline{G}-x_i$ that do not contain $x_j$. Let $R_3$ be the subgraph induced on $V(G)-(V(R_1)\cup V(R_2) \cup \{x_1,x_2\})$. We first consider the case when $V(R_3)$ is empty. We may assume that $|V(R_1)| \geq |V(R_2)|$ and so $|V(R_1)|\geq 3$. Note that if $|V(R_2)| \geq 2$, then there is a vertex $r$ in $R_1$ such that $G-r$ is $2$-connected, a contradiction. Therefore $|V(R_2)|=1$ and so $|V(R_1)|\geq 4$. Let $\beta$ be a green neighbour of $x_1$ in $R_1$. Note that $G-r$ is $2$-connected for every vertex $r$ in $V(R_1)-\beta$, a contradiction. Therefore $V(R_3)$ is non-empty. Observe that, if both $R_1$ and $R_2$ have at least two vertices, then $G-r$ is $2$-connected for any vertex $r$ in $R_3$, a contradiction. Therefore we may assume that $|V(R_1)|=1$. We show that neither $R_2$ nor $R_3$ has more than two vertices. Assume that $R_i$ has more than two vertices for some $i$ in $\{2,3\}$. Then there exists a vertex $r$ in $V(R_i)$ such that both $x_1$ and $x_2$ have at least two green neighbours in $G-r$. Note that $G-r$ is $2$-connected, a contradiction. Therefore $|V(R_2)|= |V(R_3)|=2$. Observe that there is a vertex $r$ in $R_3$ such that both $x_1$ and $x_2$ have green degree at least two in $G-r$. It follows that $G-r$ is $2$-connected, a contradiction. Thus $\overline{G}$ has three cut vertices. 

 Let $X=\{x_1,x_2,x_3\}$ be the set of cut vertices of $\overline{G}$. We may assume that for the cut vertex $x_1$ of $\overline{G}$, the components of $\overline{G}-x_1$ can be partitioned into subgraphs $P$ and $Q$ such that $x_2$ is in $P$ and $x_3$ is in $Q$, and $|V(P)|\geq |V(Q)|\geq 2$. Note that all vertices in $P$ are green neighbours of $x_3$ and all vertices in $Q$ are green neighbours of $x_2$. If $|V(P)|\geq 4$, then there is a vertex $r$ in $P$ such that all vertices in $X$ have at least two green neighbours in $G-r$ and so $G-r$ is $2$-connected, a contradiction. Therefore $|V(P)|=|V(Q)|=3$. Note that there is a vertex $r$ in $P \cup Q$ such that all vertices in $X$ have at least two green neighbours in $G-r$ and so $G-r$ is $2$-connected, a contradiction. Thus \ref{bound_1} holds. 

By Lemma \ref{small_sesquicographs}, it is clear that $|V(G)|\geq 5$ and so $|V(G)|$ is either $5$ or $6$. Suppose $|V(G)|=5$. Since $\overline{P_5}$ is the only graph on five vertices that is not critically $2$-connected, has vertex connectivity two, and whose complement has vertex connectivity one, by Lemma \ref{other_minimal_graphs}, we have $G \cong \overline{P_5}$. Next suppose that $|V(G)|=6$. Implementing the algorithm in Figure \ref{case_6} in Sagemeth \cite{sage}, it can be easily checked that $G$ is isomorphic to one of the graphs in $\{H_1, H_2, H_3, H_4, H_5\}$. This completes the proof.
\end{proof}

\begin{proof}[Proof of Corollary \ref{main_corollary}]
Note that every cycle of length exceeding five has the cycle of length five as an induced minor. Also, the domino graph $C_6^+$ contains $\overline{P_5}$ as an induced minor. The result now follows by Theorem \ref{main}.
\end{proof}

\section*{Acknowledgement}

The author thanks James Oxley for helpful suggestions. The author also thanks Thomas Zaslavsky for helpful discussions.


\begin{thebibliography}{99}











\bibitem{corneil} D.G. Corneil, H. Lerchs, L. Stewart Burlingham,  Complement reducible graphs, {\em Discrete Appl. Math.} \textbf{3} (1981), 163--174.

\bibitem{corneil2} D.G. Corneil, Y. Perl, L.K. Stewart, A linear recognition for cographs, {\em SIAM J. Comput.} \textbf{14} (1985), 926--934.


\bibitem{text} R. Diestel, {\em Graph Theory}, Third edition, Springer, Berlin, 2005.


\bibitem{jung} H.A. Jung, On a class of posets and the corresponding comparability graphs, {\em J. Combin. Theory Ser. B} \textbf{24} (1978), 125--133 .

\bibitem{nauty} B.D. McKay and A. Piperno, Practical graph isomorphism II, {\em J. Symbolic Comput.} \textbf{60} (2014), 94--112.

\bibitem{nebesky} L. Nebesky, On induced subgraphs of a block, {\em J. Graph Theory} \textbf{1} (1977), 69--74.

\bibitem{oxjag} J. Oxley and J. Singh, Generalizing cographs to 2-cographs, arxiv: 2103.00403.






\bibitem{sage} SageMath, the Sage Mathematics Software System (Version 8.2), The Sage Developers, 2019, http://www.sagemath.org.

\bibitem{sein} D. Seinsche, On a property of the class of $n$-colourable graphs, {\em J. Combin. Theory Ser. B} \textbf{16} (1974), 191--193 .


\bibitem{west} D.B West, {\em Introduction to Graph Theory}, Second edition, Prentice Hall, Upper Saddle River, N.J., 2001.







\end{thebibliography}
\end{document}